\documentclass[runningheads,a4paper]{llncs}
\usepackage{amssymb}
\usepackage{graphicx}

\usepackage{url}

\newtheorem{thm}{Theorem}
\newtheorem{deff}[thm]{Definition}
\newtheorem{cor}[thm]{Corollary}
\newtheorem{lem}[thm]{Lemma}

\mathchardef\hy="2D

\newcommand{\perm}[2]{\left\{ \begin{array}{l} {#1} \\ {#2} \end{array}   \right.}
\newcommand{\prmb}[2]{\left\{ \begin{array}{ll} {C_1:}  & {#1} \\  {C_2:}  &  {#2} \end{array}   \right.}

\begin{document}

\mainmatter

\title{How to draw combinatorial map? From graphs and edges to corner rotations and permutations}
\titlerunning{How to draw combinatorial map?}
\author{Dainis Zeps, Paulis \c{K}ikusts}
\authorrunning{ Dainis Zeps, Paulis \c{K}ikusts}
\institute{Institute of Mathematics and Computer Science, \\ University of Latvia, \\
29 Rainis blvd., Riga, Latvia\\
\url{dainize@mii.lu.lv}\\
\url{paulis@mii.lu.lv
}}

\toctitle{How to draw combinatorial map?}
\tocauthor{Dainis Zeps, Paulis \c{K}ikusts}

\maketitle

\begin{abstract}
In this article we consider combinatorial maps approach to graphs on surfaces, and how between them can be establish terminological uniformity in favor of combinatorial maps in way rotations are set as base structural elements and all other notions are derived from them. We set this approach as rotational prevalence principle. We consider simplest way how to draw combinatorial map, and ask how this approach in form of rotational prevalence could be used in graphs drawing practice and wider in algorithms. We try to show in this paper that the use of corners in the place of halfedges is much more natural than that of halfedges. Formally there is no difference between both choices, but corner approach is much more clear and concise, thus we advocate for that.
\end{abstract}

\section{Introduction}

We deal in this article with some simple considerations and observations  of how to draw combinatorial map, and how it comes in  connection with traditional drawing of graphs, fig. \ref{fig0}, see fig. \ref{fig3} in \ref{graph.surface} too.

\begin{figure}[h]
\begin{center}
\scalebox{.25}{\includegraphics{map.md}}
\scalebox{.45}{\includegraphics{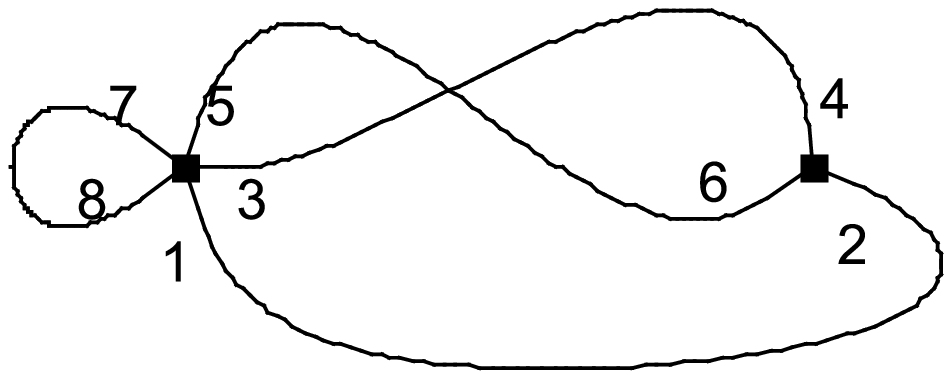}}
\caption{Images of the same combinatorial map drawn in the plane.}
\label{fig0}
\end{center}
\end{figure}.

To draw either graph or combinatorial map doesn't much differ if both require to come to picture of points and lines confining areas in the plane. Combinatorial map's approach suggest its own natural way how to come to its image and we try to follow directly this line. Moreover, we aim to build some base for graph's image what would follow from rotational approach, and, even more, establish it as sort of principle, calling it rotational prevalence when we go out from assumption that rotations in the graph defined as combinatorial map might serve as starting point for all other notions around the topology of the graph. We apply this approach in our case of building combinatorial map's drawer, accompanying with some questions  what would be helpful for graph drawing environments.

Combinatorial map theory is new area of combinatorics which development may have considerable impact on topological graph theory. Besides, combinatorial functions around these maps are both theoretical tools and effectively calculable means that could be developed in complex environments for applications. These directions, as it seems to us, are weakly developed, partially maybe due to fact that combinatorial maps and apparatus around them as if repeat all what is already present in topological graph theory. But combinatorial maps as branch of combinatorics may be investigated independently from graph theory using e.g. such simple tools as permutations, and use whatever results already achieved in e.g. permutational group theory. Astonishing results of application of constellations (what is some generalization of combinatorial maps) in Riemann geometry \cite{lz 04}  show that this area deserve to be researched with much more effort than before.

Combinatorial maps invention is attributed to  Tutte \cite{tu 73,tut 84}, developed by Stahl \cite{sta 80,sta 83,sta 88} and many other researchers, see e.g. \cite{lit 88,bo li 95,liu 94}. Combinatorial map approach is based on discoveries of rotations in graphs, see Heffter \cite{he 1891,heffter}, Edmonds \cite{edm 60} and Jacques \cite{ja 70}. These facts are researched by many, see e.g. \cite{gibl 77,whi 73,rin 74,moh 01}. One way to introduce combinatorial map is by pair of permutations, see \cite{bo li 95,lz 04}. Lando and Zvonkin in their book "Graphs on Surface and Their Applications"  \cite{lz 04} points at Jacques \cite{ja 70}, who first used these constructions under the name of constellations. Constellations as generalization of combinatorial maps are used by Lando and Zvonkin in Riemann geometry \cite{lz 04}, see \cite{pa zv 02} too. We do researches in these area since 1993 \cite{ze 93,kize,ze 94,ze 96,ze 97,ze1 97,ze 98,ze 99,ze 08,ze 10}, that is based on permutational representation of rotations.

Specifically, combinatorial maps as purely combinatorial structure could deserve wider application in graph theory and some of its applications as e.g. graph drawing, where combinatorial maps could serve as more basic structure than graph itself. Rotations could be the combinatorial objects on which all other graph-theoretical invariants could be represented. In this article we try to reconstruct such option with questioning how to implement these observations more directly.

We tried to built combinatorial map drawing tool where we used directly combinatorial map approach to depict its eventual visual image, see fig. \ref{fig3} in \ref{edge}, in the same time giving account to the fact that all programming tradition used by graph drawings is based on rather far from rotational way of thinking in combinatorics. The question we post in this article is  - Can {\em rotations before edges and vertices} approach be applied directly in graph drawings particularly with some sensible outcome that could serve for applications?

By developing tools for combinatorial map drawing and by giving ways to the visualization of combinatorial maps we hope  to widen the use of  them in topological graph theory.

\section{From graph on surface to set of rotations in the graph}

For purposes of this paper we consider combinatorial map as arbitrary pair of permutations.

Combinatorial map as purely combinatorial object naturally models graph (in general hypergraph) embedded into oriented surfaces. By discovery of Heffter \cite{he 1891,heffter} and reminder of Edmonds \cite{edm 60} we know that vertex rotation, i.e., rotational order of edges around vertex, taken for all vertices, fixes graph on some orientable surface. But, if we fix rotational order of edges in faces too, than these two rotations, vertex rotation and face rotation, are sufficient to code all the graph, i.e., no sets of vertices and edges are to be specified, because these two rotations already cipher the graph on some orientable surface.

Moreover, coding these rotations with permutations and performing operations within purely permutational calculus, we may maintain all operations upon graph in the same way, i.e., using operations on permutations, that usually is done in vertices-edges-faces operational framework.

There is very natural way how to come from vertices and edges  to permutational setting. Let graph $G=(V,E)$ be fixed by vertex rotation. If we depict each unoriented edge as pair of oriented edges in a way that outgoing edge comes first before incoming edge in rotation of edges around vertices in clockwise direction, we observe that oriented edges around faces are oriented all in one direction, and this direction is anticlockwise. See fig. \ref{fig1} left. Now we are to do one more abstract operation -- we replace the corner of the face that was formed by two oriented edges by halfedges, i.e., incoming head and outgoing tale, and call this new object {\em corner}. By the way, doing this, trivially, we replace double cover of borders of faces with corners of oriented edges with simple cover of corners of halfedges. See fig. \ref{fig1}.

\begin{figure}[h]
\begin{center}
\scalebox{.55}{\includegraphics{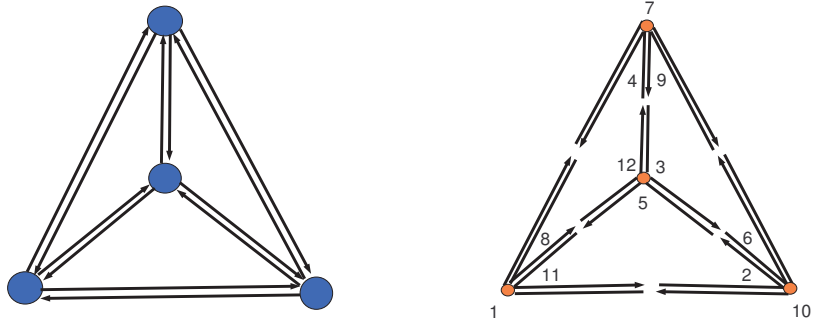}}
\caption{Graphs $K_4$ embedding in plane fixed by rotation of alternatively outgoing and incoming edges around vertices, left, and by vertex rotation and face rotation of corners of edges, right. Numbering corners we come to permutations, which stand for rotations. Corners around vertices represent permutation $(1 8 \bar{1})(2 6 \bar{0})(3 5 \bar{2})(4 7 9)$ and around faces permutation $(1 7 \bar{0})(2 5 \bar{1})(3 6 9)(4 8 \bar{2})$. Two edge rotations are respectively $(1 2)(3 4)(5 6)(7 8)(9\bar{0})(\bar{1}\bar{2})$ and $(1 4)(2 3)(5 8)(6 7)(9\bar{2})(\bar{0}\bar{1})$.}
\label{fig1}
\end{center}
\end{figure}

These corners are as many as oriented edges in the graph, i.e., double as many as edges in the graph. Of course, these corners cover all oriented edges only once. Now, by fixing vertex and face rotations we actually fix two corner rotations, correspondingly around vertices and faces. The size of both rotations is the same, that of number of corners. Actually two more rotations of the same size appear, i.e., two edge rotations, which we get if we establish new adjacency of corners across what former was edge. Edge rotations in graphs have only 2-cycles, thus corresponding permutations are involutions. In case of hypergraphs edge rotations may have orbits of arbitrary size, i.e., corresponding to size of hyperedge.

Now, in case of graphs with possibly hyperedges we have four rotations of equal size that fix this graph on surface and without any necessary additional information, of course. But, four rotations are redundant. Both edge rotations have the same cycle structure, or {\em passport} \cite{lz 04}. Besides, taking only three rotations, but all in one direction, say, anticlockwise, multiplication of them is equal to identity permutation, \cite{lz 04}. (For that reason combinatorial map is 3-constellation.) This means that third permutation always may be calculated from the two given. Now, in case of  graphs without hyperedges edge rotations are involutions without fixed points with respect to number of corners.  If we fix edge rotation once and for ever, because its passport is constant, we may calculate all permutations from only one given permutation except two edge rotations which don't fix particular graph but some larger class. We further should see that this class is the class of the knot, to which this graph belongs.

\subsection{From operations with graphs on surfaces to permutational calculus}

With having corners ($C$) introduced we directly and naturally come to permutational calculus. Corners are as many as oriented edges in the graph, so number of corners $\|C\|(=m)$ is equal to $2\|E_G\|$, i.e., for graphs without hyperedges this number is always even. Thus, corners we may label (or equate at least for practical purposes) with natural numbers from interval $[1..m]$. We have four permutations corresponding to four rotations of the same size $m$. By the way, identical permutation of degree $m$ as vertex rotation should have  graph consisting from $m/2$ isolated edges in correspondence.

\subsection{Hypergraphs and pairs of permutations}

By the way, for graphs we should have only even permutations, but in case of hypergraphs arbitrary permutations come before. To fix hypergraph we need three permutations, or, by fixing edge rotation, both vertex and face rotation. In general, two permutations always have some hypergraph in correspondence\cite{ze 97}. In \cite{ze 97} we call combinatorial maps corresponding to hypergraphs {\em partial maps} in opposition to {\em graphical maps} for corresponding to graphs without hyperedges. The name of partial map is motivated by the fact that partial map may be considered as map with cut out some faces, thus being as if partial with respect of that graphical map with restituted cut out faces.

\begin{figure}[h]
\begin{center}
\scalebox{.30}{\includegraphics{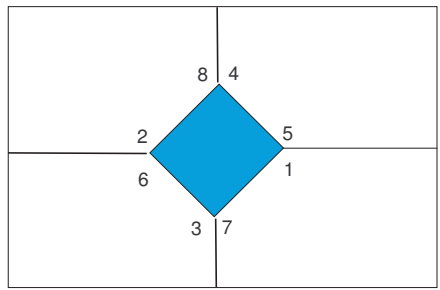}}
\caption{Example of partial map with vertex rotation $P=(1 5)(2 6)(3 7)(4 8)$, face rotation $Q=(1 7 45 2 8 9 6)$ and edge rotation $R=(1 4 2 3)(5 6)(7 8)$. One of edges, $(1 4 2 3)$, is not graphical but hyperdedge. See cover of two books \cite{lz 04} and \cite{ze 07}.}
\label{fig2}
\end{center}
\end{figure}

See in fig. \ref{fig2} example of hypergraph, graphs $K_4$ embedding on torus with cut out face. The cut out face stands for hyperedge of degree four. That we are to do with hypergraph we get from the fact that edge rotation is not graphical.

\section{Graphical combinatorial maps introduced}

 Now we are ready to start combinatorial map calculus, which are based on permutational calculus, where permutations act on elements of set $C$, that stand for corners earlier introduced.

 We multiply permutations from left to right. Graphical {\em combinatorial map} is pair of permutations,  $(P,Q)$, with $P \cdot Q^{-1} (=\rho)$ being involution (i.e., orbits only of length two) without fixed elements. $P$ and $Q$ are correspondingly {\em vertex rotation} and {\em face rotation}. For graphical combinatorial map we distinguish two edge rotations,  {\em inner edge rotation} $\pi=Q^{-1}P$ and {\em edge rotation} $\rho=P \cdot Q^{-1}$. We assume permutations acting on set of elements $C$, usually natural numbers from $1$ to $m$, $m=\|C\|$. We try to work within set of maps with fixed  $\pi$  calling them {\em normalized maps}. Mostly we use one particular choice of $\pi$ equal to $(1 2)$ ... $(2k-1$ $2k)$, $k\geq 1$. The same convention is used in \cite{lz 04}. Under these assumptions, graphical combinatorial map is characterized by one permutation or one rotation, say vertex rotation $P$.

 \subsection{One-one correspondence between permutations and graphs on surface}

 The fact we come to is worth to be appreciated specially. Two rotations fix combinatorial map, but under convention that inner edge rotation is fixed for some whole class of combinatorial maps it suffices only one permutation to specify graphical combinatorial map, see \ref{cmcsgs} lower. This feature is preserved under multiplication of permutations, as it is shown in \cite{ze 96}. As the result most of operations may be done within this convention.

 Permutational calculus and theory around permutations usually coming under name of permutational group theory is very developed, e.g. \cite{wie 64}. What does it mean that graphs on surface one-one correspond to permutations? One way to perceive this is to state that whatever theorem in permutations work equally in this class of graphs. One of our aims of this article is to turn attention on how to maintain this permutational calculus available in graph theoretical applications, thus giving eventual access to whatever useful in theory around permutations.

 However, there may appear applications where we need to go outside the class of normalized maps, but it is easy to bookkeep these situations, but use, whenever possible, benefits of normalized maps.

 \subsection{Drawing graph on surface}\label{graph.surface}

 A graph with loops and multiedges on an orientable surface corresponds to arbitrary graphical combinatorial map in very natural way. One intuitively well based way to persuade oneself about this is to draw directly this graph in the plane in the following way. Let first put as many points in the plane as orbits in the vertex rotation with {\em edge-ends}  clockwise around with corners following in cyclical order for each corresponding point, see fig.1, a. Further, let us unite two edge-ends with corners (following clockwise) $a$ and $b (=a^{\pi})$. I.e., corners that form orbit of inner edge rotation $\pi$ are to be united with an edge in the drawing, thus  justifying an option to speak about {\em edges of combinatorial maps}. It is easy to see that changing orientation from clockwise to anticlockwise we had to use edge rotation $\rho$ in place of inner edge rotation $\pi$. Thus, two opposite directions of rotations give two distinct edge rotations both in drawings and combinatorially.

\begin{figure}[h]
\begin{center}
\scalebox{.55}{\includegraphics{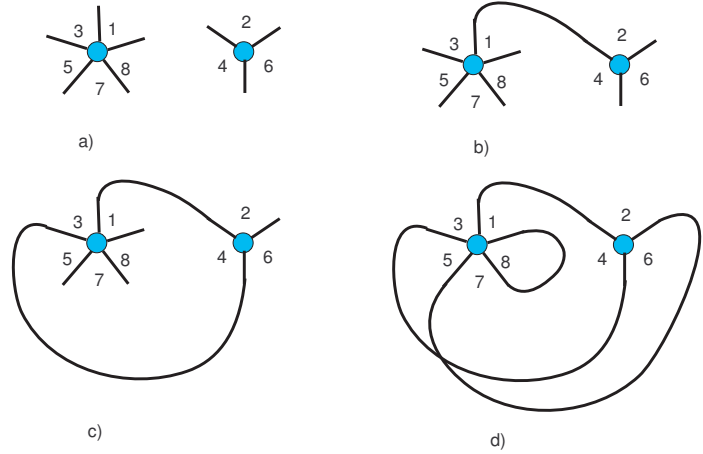}}
\caption{a) Putting two points in the plane with edge-ends and corresponding corners from orbits of the given permutation $(1 8 7 5 3)(2 6 4)$. b) Uniting edge-ends with corners 1 and 2. Notice the order edge-end and the corresponding corner in the clockwise direction. c) Uniting edge-ends 3 and 4. d) Uniting two left edge-end pairs. Notice that last "edge" of the map is loop.}
\label{fig3}
\end{center}
\end{figure}

 See e.g. fig. \ref{fig3} where combinatorial map  $P=(1 8 7 5 3)(2 6 4)$ is drawn. In a) we place points in the plane corresponding to orbits in vertex rotation. In b) to d) to get drawing of combinatorial map, we are to unite pairs of edge-ends $(2k-1, 2k)$ with curves for edges from $k= 1$ to $4$. In d) we have the drawing fulfilled. Notice that loops correspond to edge-ends that go out from the same point in the plane.

 Further examples of combinatorial map drawings performed with programmed map drawer see in section \ref{drawer}.

\subsection{Definition of edge of combinatorial map}\label{edge}

Higher we started to speak about {\em edge of combinatorial map} in some informal, intuitively based way.

Thus, we are motivated to consider  the edge of combinatorial map in some more formally way.

Let $\Omega_p$ and $\omega_p$ be correspondingly set of orbits and some orbit of  permutation $p$.

\begin{deff} Let for a given map $P,Q,\pi,\rho$ sextet e= $\langle \omega^1_P, \omega^2_P; a, b, c, d  \rangle$ be such that $a, b \in \omega^1_P$, $c, d \in \omega^2_P$, $(a c) \in \Omega_{\pi}$, $(b d) \in \Omega_{\rho}$ and $(a c)^P=(b d)$. We call $e$ edge of combinatorial map.
\end{deff}

Let us use some denotations for transitive permutations in cyclical form. For one corner orbit $(a)$ we use small letter, i.e., $a$, for arbitrary length orbit (or part of orbit) we use capital letter, i.e., A might be orbit or its part $(a_1, ..., a_k)$, $k \geq 0$.  In this manner $aA$ denotes some transitive permutation $p$ in cyclical form with corner $a$ being distinguished and rest part being $A$.

Using this convention we may write an edge of a map in the form $$\langle baA, dcC; a (=c^{\pi}), b (=d^{\rho}), c (=a^{\pi}),  d (=b^{\rho}) \rangle.$$ Let us fix this as a lemma.

\begin{lem} Orbits $(a c)$ and $(b d)$ correspondingly of inner edge rotation and edge rotation have edge $$\langle baA, dcC; a (=c^{\pi}), b (=d^{\rho}), c (=a^{\pi}),  d (=b^{\rho}) \rangle$$ of combinatorial map in correspondence with corresponding equalities being held.
\end{lem}

We must observe that the edge we speak about turns actually into a loop in case orbits $baA$ and $dcC$ coincide. But corners $a,b,c,d$ all should be distinct nevertheless except in cases of isolated or hanging edges or loops. So isolated edges or loops has two distinct corners, $\langle a, b; a, a, b, b \rangle$ being isolated edge, and $\langle ab, ab; a,b,b,a \rangle$ being isolated loop. Hanging edge or loop has three distinct corners, e.g., $\langle ba, c; a,b,c,c \rangle$ being hanging edge, and $\langle abc, cab; a,b,c,b \rangle$ being hanging loop.

E.g., for normalized map $P=(1 8 7 5 3)(2 6 4)$ $e=\langle (1 8 7 5 3), (2 6 4); 1, 3, 2, 4 \rangle$ is an edge, comp. fig. \ref{fig3}. In case we don't want to specify rotations from where edges go out we use shorter form, quartet notation only for four corners, e.g., we write for edge $e=\langle 1, 3, 2, 4 \rangle$.

\subsection{Graphs on surfaces}

Usually treating graphs on surface we start from graph $G=(V,E)$ and then supply it with cyclical succession of neighboring edges, i.e., with some permutation, say, $\Pi$ that works on set of unoriented edges. It turns out that graph on surface might be considered as simpler structure than graph supplied with rotations, i.e., rotations of vertices and faces already fix all graph-on-surface structure without taking much concern on what kind on surface "graph" was "embedded".

This may be adequately taken into account by fixing graph on surface by {\em corners between neighboring edges} in rotations. Let graph $G$ be fixed by edge rotations around vertices. Let C be set of corners between edges and two rotations (or actually  permutations) are given: rotation of corners around vertices $R_V$ and rotation of corners around faces $R_F$. It is easy to see that now some pairs of corners $(c_1, c_2)$ are {\em connected via edge} in the way that $c_1^{R_V^{{R_F}^{-1}}}=c_2$. This fact is sufficient and crucial to establish one-one correspondence between graphs on surface and combinatorial maps.

\begin{lem} For a graph $G$ on surface with given rotations of corners of neighboring edges around vertices and faces $R_V$ and $R_F$ the pair of permutations $(R_V,R_F)$ is graphical combinatorial map.
\end{lem}
\begin{proof} To prove this assertion it suffices to show that multiplication $R_V \cdot R_F^{-1}$ is involution without fixed points. But for arbitrary corner $x$ its {\em edge-adjacent} corner $x^{R_V^{{R_F}^{-1}}}=y$ is some other corner at some vertex and applying backwards we get that edge-adjacent corner to $y$ is $x$, i.e., $y^{R_V^{{R_F}^{-1}}}=x$. But this means that multiplication $R_V \cdot R_F^{-1}$ is involution without fixed points.
\end{proof}

To conclude we state that {\em edge of combinatorial map} $(R_V,R_F)$ may be interpreted as edge of graph $G$ on surface fixed by this pair of rotations $(R_V,R_F)$.

\begin{cor} For combinatorial map $(R_V,R_F)$ corresponding to graph $G$ on surface sextet
$$\langle baA=\omega_{V_G}, dcC=\omega_{V_F},a=c^{{R_V^{-1}}^{R_F}}, b=d^{R_V^{{R_F}^{-1}}},c=a^{{R_V^{-1}}^{R_F}}, d=b^{R_V^{{R_F}^{-1}}} \rangle$$
 is map's edge that has edge $(b,d)$ uniting corners $b$ and $d$ of graph $G$ on surface in correspondence.
\end{cor}

\subsection{Combinatorial map as combinatorial structure of graph on surface}\label{cmcsgs}

We may establish one-one correspondence between graphs on surfaces and graphical combinatorial maps. To get it we are to take set of corners $C$ that should serve as set of elements of combinatorial maps. Further, we fix as arbitrary graph on surface pair of rotations, for vertices and faces correspondingly, say, $R_v (=P)$ and $R_F (=Q)$. Now, inner edge rotation should be $\pi = P^{-1} \cdot Q$ and edge rotation $\rho=P \cdot Q^{-1}$. Now, combinatorial map base assumptions says that graph on surface is well defined whenever edge rotation $\rho$ is involution without fixed points. This means that under these assumptions we have established one-one correspondence between graphs on surfaces and combinatorial maps.

In \cite{ze 96} we saw that we may fix one of edge rotation, say, inner edge rotation and work, whenever possible, within this constraint. Rightly, we may get along with one edge rotation, second edge rotation actually as variable being superfluous. We chose to fix inner edge rotation, and, by the way, in a very convenient way, fixing it $(1 2)(3 4)(5 6)... (2m-1 2m)$. Such maps we call normalized maps. Further, in \cite{ze 10} we see that we may go on with normalization and normalize knot in the map too. Then corresponding graphs on surfaces behave in complete correspondence with maps and we don't loose one-one correspondence between maps and graphs on surfaces.

All this taking into account we are not to be too specific to specify always what we have in mind when we speak, say, about corners, or some else aspect, either corresponding to maps or graphs on surfaces: we always use map-graph-on-surface correspondence.

\subsection{Zig-zag walk and the knot}

Let us perform zig-zag walk in the graph on surface. Taking corner $c$ let us go on with applying $\pi$ and $\rho$ alternatively and we are to return to corner $c$, see \cite{ze 10}. I.e., this routine is cyclical. If we haven't exhausted all corners we take some free one and go on. In this way we get another rotation for a graph on surface, or map, that we call zig-zag walk for graphs on surface, and knot for maps. However, this rotation is not unique, i.e, it has $2^k$ possible values for $k$ orbits in the knot. It is easy to see that each change of orientation of some orbit of knot gives another well defined value of knot.

If we fix one particular value of knot we actually divide set of corners $C$ into two equal size parts $C_1$ and $C_2$ so that we may express this particular knot as  $\mu=\prmb \pi \rho$, where the expression $\mu=\prmb p q$ say, that within set $C_1$ we apply $p$, and within $C_2$ we apply $q$.

Let us color corners into two colors according their belonging either to $C_1$ or $C_2$. Usually we use for $C_1$ green color and for $C_2$ red color.

 Particular choice of knot $\mu$ partitions $\pi$ into $\pi_1 \cdot \pi_2$, where we call $\pi_1$ {\em cut edges} and $\pi_2$ {\em cycle edges}.   For {\em cut edge} $\langle a, b, c, d \rangle$ corners $a, b$ receive one color and $c, d$ another color. Correspondingly, for {\em cycle edge} $\langle a, b, c, d \rangle$ corners $a, d$ receive one color and $b, c$ another color.   Thus we have that $P \cdot \pi_1:C_1\mapsto C_2$ or $P \cdot \pi_1:C_2\mapsto C_1$ , i.e., changes color, and $P \cdot \pi_2:C_1\mapsto C_1$ or $P \cdot \pi_2:C_2\mapsto C_2$, i.e., retains color. In \cite{ze 96} was shown that by fixing $\mu$ map $P$ may be expressed as multisplication $\gamma_1 \cdot \gamma_2 \cdot \pi_2$, where $\gamma_1$ acts within $C_1$ and $\gamma_2$ acts within $C_2$. From this we get formula for knot $\mu$. $\mu=\gamma_2 \cdot \pi \cdot \gamma_1^{-1}$, see \cite{ze 10}. Finally, in general we get $$\mu=\prmb \pi \rho=\perm {C_2:\gamma_2 \pi \gamma_1^{-1}} {C_1:\gamma_1 \rho \gamma_2^{-1}}.$$

 One might ask what is knot for graph on surface? In \cite{ze 10} we show that actually the equivalence class of zig-zag walks is graph's on surface invariant. To come closer to the knot and what it means we may observe that partial map $(P, \mu)$ has edge rotation $Q$, i.e., this partial map, which face rotation is $\mu$, is hypergraph with edges these of faces of original map. This fact says that fixing graph $G$ on surface fixes another hypergraph with the same vertex rotation but hyperdeges these of faces of $G$, in which case faces of this hypergraph should be knot of the graph.

 \subsection{Knotting of the map}

Fixing knot $\mu$ for map $P$ we get by multiplying permutations new rotation for maps $\alpha$, i.e., $$P=\mu \cdot \alpha.$$
We call this rotation $\alpha$ knotting of map because of its peculiar features. Knotting belongs to selfconjugate class of maps, i.e., $\alpha^{\pi}=\alpha$, which form a group against multiplication\cite{ze 96}.

Knotting $\alpha$ is responsible for edge structure of maps\cite{ze 10}. To see this we define knotting's symmetric form $A=\alpha \cdot \pi_1=\gamma_1 \cdot \gamma_1^{-1}$, involution $\delta=\pi^{\gamma_1}$ and {\em edge structuring knot} $\varepsilon=\prmb \pi \delta$. In \cite{ze 10} we show that $\varepsilon^2\cong A$, i.e., one of square roots of knotting's symmetric form is equal to edge structuring knot.

\subsection{Normalizing knots}

In \cite{ze 96} we considered class of maps with fixed knot. Actually this class has fixed both inner edge rotation $\pi$ and edge rotation $\rho$. Such class may be called class of maps with normalized knot. We add one more restriction: normalized knot should have corners in augmented order.

In fig. \ref{fig6} we see map with normalized knot. Odd corners have received one color and even corners another. But to get this feature we don't need to knot be normalized. It suffices for it to be only {\em partially normalized}, i.e., all edge should appear in the knot in the same order, that is, for all edges its less valued corner should come before, or reversely. In \cite{ze 10} we show  for knot and edge structuring knot that if one of them is normalized then other is partially normalized.

\subsection{Combinatorial adjacency of corners}

By drawing graph we use vertices, edges and faces. Replacing these three types of objects with rotations we retain minimally necessary for what is the picture of the graph in the essence. One more way to see this clearly is to notice that combinatorial maps may be treated as adjacency relations of corners of halfedges of three types, vertex, face and edge adjacencies, see fig.\ref{fig4}. Taking corners of two halfedges we may assemble them into rotation in three ways, see fig. \ref{fig5}: a) around common apex of corners, forming vertex rotation, b) by adjacing halfedges in way forming face rotation, and c) by flipping one corner to its mirror image and adjacing halfedges in way forming an edge rotation. In this we may replace vertices, edges and faces with adjacencies of corners of three types, and do just the same in replacing graphs with rotations in combinatorial maps.  Let us name these three adjacencies correspondingly p-adjacency, q-adjacency and r-adjacency.

\begin{figure}[h]
\begin{center}
\scalebox{.35}{\includegraphics{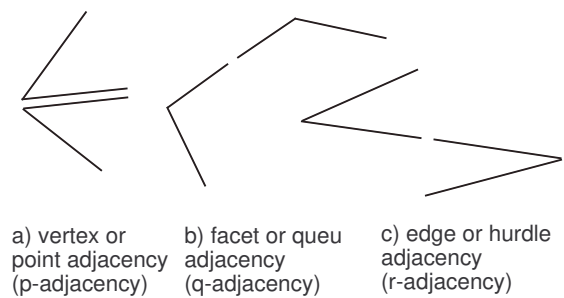}}
\caption{Combinatorial adjacency of corners. }
\label{fig4}
\end{center}
\end{figure}

Let us say that map is fixed by p-adjacency and q-adjacency if rotations $P$ and $Q$ are given. Thus, normalized maps are fixed by p-adjacency or q-adjacency, r-adjacency being fixed for whole class of maps. Knot appear to us as clockwise anticlockwise alternation of r-adjacency, or {\em left-right r-adjacency}. Edge structuring knot is as much as possible right along forwarded edges alternation of r-adjacencies. But this last in its name fixed behavior of edge structuring knot reveals itself only in case it is at least partially normalized.

\begin{figure}[h]
\begin{center}
\scalebox{.40}{\includegraphics{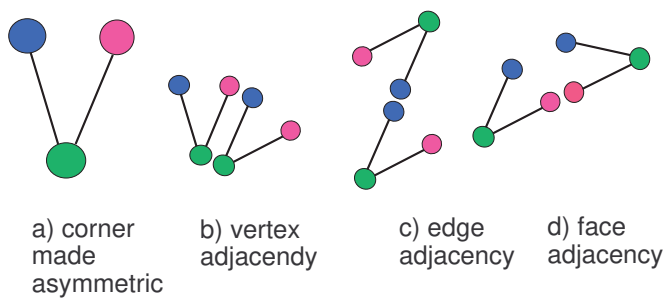}}
\caption{Combinatorial adjacency of corners symmetrized with asymmetric corner. Asymmetry in corner made by marking vertex or red vertex, edge or blue vertex and face or red vertex. Uniting in rotations green, blue and red vertices we get correspondingly vertex, edge and face adjacency. Notice that in face adjacency one corner is flipped by mirror reflection.}
\label{fig5}
\end{center}
\end{figure}

In \cite{lz 04} similar result is achieved via canonical triangulation of faces of the graph and getting three involutions, see Construction 1.5.20 on page 49, \cite{lz 04}. We may do the same by making our corner asymmetric by labeling it with three vertices of different color, apex with green label, calling it vertex's vertex, left to apex with blue label, calling it edge vertex, and right to apex with red label, calling it face vertex. Now we make three color rotations, green rotation for vertex, blue rotation for edges and red rotation for faces. Now we have come to complete symmetry, what complies with symmetry in rotations in combinatorial map.

By the way, all considerations in this chapter persuades us that the choice to number corners in the graph picture to get combinatorial map is as legitime operation as replacing the drawing of the graph with combinatorial map.

\subsection{Constellations  as combinatorial structure of Riemann surfaces \cite{lz 04,pa zv 02}}

Rotational approach gives direct generalization from Euclidean plane to ramified  Riemann sphere \cite{lz 04,pa zv 02} where we are to imagine graph being drawn on oriented surface $X$ of arbitrary genus that serves as domain of ramified covering $f$ of sphere $S^2$, $f:X \rightarrow Y$.

How to use  it here, and what it could mean for us here? This says that rotational structures are more basic and before choice of where and how we are to perform our picture's structure itself. We are as if before choice either to draw one nontrivial pictures on trivial surface, or many trivial pictures on sheets of nontrivial surface. But maybe, for technical reasons, we may get used to something in the middle, i.e., to divide global picture in smaller ones each of which we put on separate sheets. The use of rotational approach is quite insensible of where we are going to choose these middle points. But this might be only one of perspectives of rotational prevalence. Thus, in projecting whatever around rotational system we are to place them in the base, both rotations, operations with them and all apparatus around them, and only after all the rest.

\section{On rotational prevalence}

All before was told with the aim to show that rotations and permutations may be basic units of some combinatorial structures. In this order we would like to speak about rotational prevalence that would suggest to use these rotational aspects whenever and however necessary according their prevalence in theoretical settings.

\subsection{Drawing maps}

\begin{figure}[h]
\begin{center}
\scalebox{.55}{\includegraphics{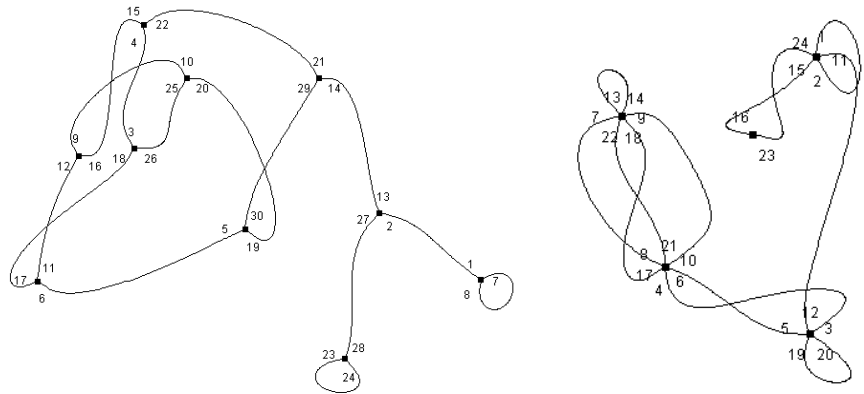}}
\caption{Two examples of randomly generated maps. Drawings may make maps look beautiful, as we see!}
\label{fig7}
\end{center}
\end{figure}

If one used to combinatorial maps comes to someone graph drawing specialist with request to draw a map the latter asks the former to prepare graphical code for his/her combinatorial map. Doesn't it occur to him/her that combinatorial map is a simpler structure than the code of the graph to be drawn? This situation is paradigmatic, and requests to be cured somehow, or at least to ask how to teach graph drawing community to learn more about rotational systems  in order to make rotational prevalence we speak about in this article as working principle.

\begin{figure}[h]
\begin{center}
\scalebox{.23}{\includegraphics{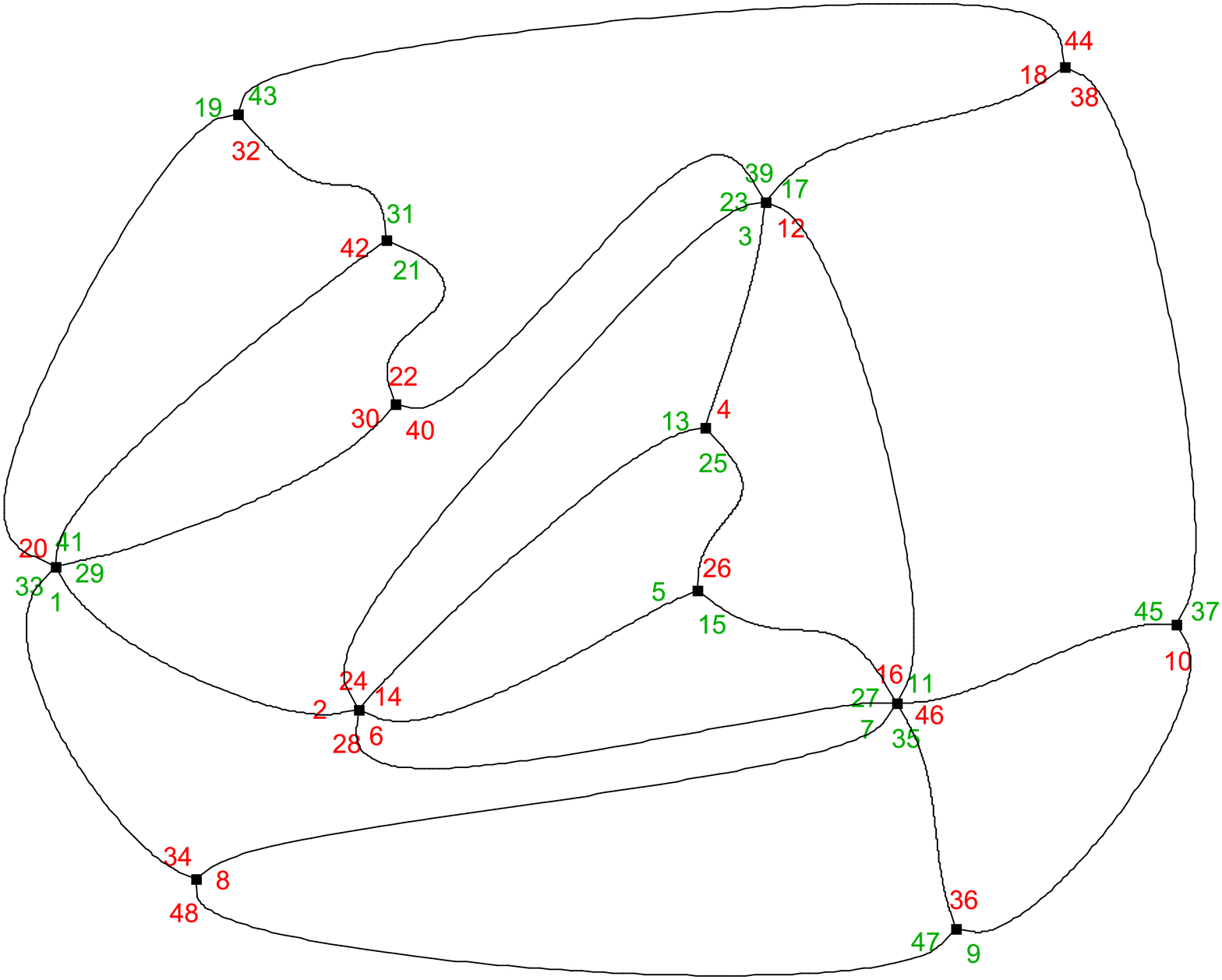}}
\caption{Example of  combinatorial map with  normalized knot. Cover of set of corners with green and red cycles clearly seen. Corresponding graph is a 4-critical multiwheel\cite{ze 12}.}
\label{fig6}
\end{center}
\end{figure}

What would be that that first had to teach to the second in case one needs to draw a graph?
One way is to work according the schema given in this article, i.e., code the graph as two rotations of corners and calculate edges of the corresponding map. It would be formal implementation of rotational prevalence. But it wouldn't give much if we came to necessity to start to implement rotations as points, lines and bordered areas in the plane, that would correspond to the elements of the pictures of the graphs we are used to. Would these last objects  to be redesigned in terms of some rotational geometry that with such easy could be done in case of Riemann geometry? In graph's rotational geometry classes of vertices, edges and faces don't differ, that is not the case for their opposites - points, lines and areas.

On the other hand, if we follow the line how rotational prevalence would work on level of theoretical considerations, we might keep as close as possible to it, and turn our attention on how to develop this ground level of rotational calculus. Possibility to visualize maps and especially partial maps should be very useful in these efforts.

\begin{figure}[h]
\begin{center}
\scalebox{.25}{\includegraphics{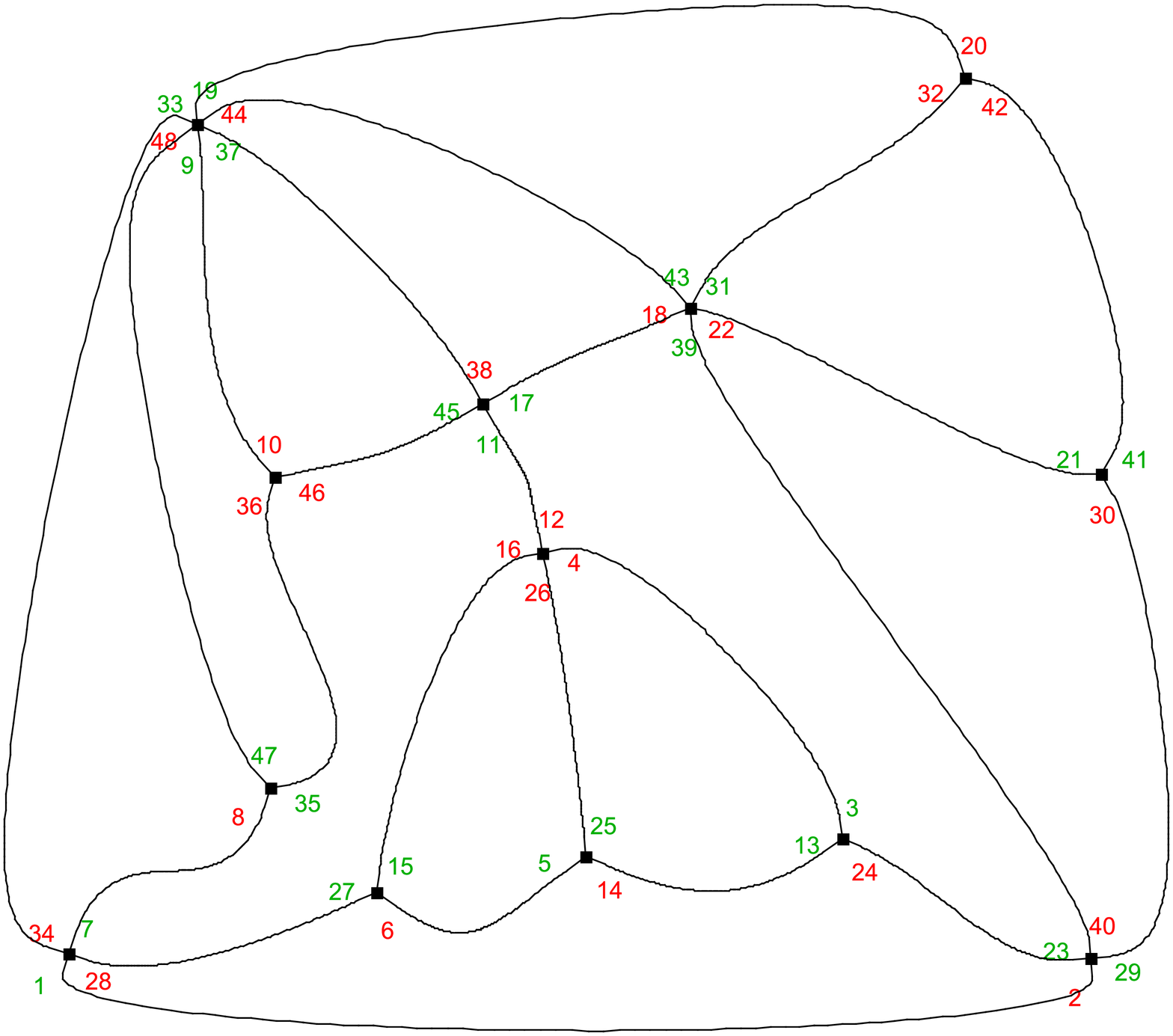}}
\caption{Dual map of previous example.}
\label{fig8}
\end{center}
\end{figure}

\subsection{Combinatorial map drawer used in this article}\label{drawer}

We have endeavored to build a small drawer for combinatorial maps using the simplest way by placing points supplied with edge-ends with corners for vertices in the plane and uniting them by edges from edge rotation according \ref{edge}. The edge ends first was united by rectangular lines, which  afterwards were smoothed into curve-like edges. The drawer was supplied with manual operation to move vertex in the plane, that was supported with corrector of all rest edge lines with respect to this relocated vertex.

\begin{figure}[h]
\begin{center}
\scalebox{.30}{\includegraphics{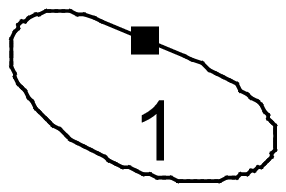}}
\scalebox{.30}{\includegraphics{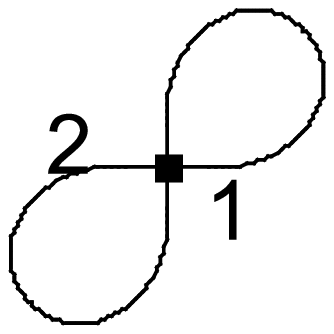}}
\scalebox{.30}{\includegraphics{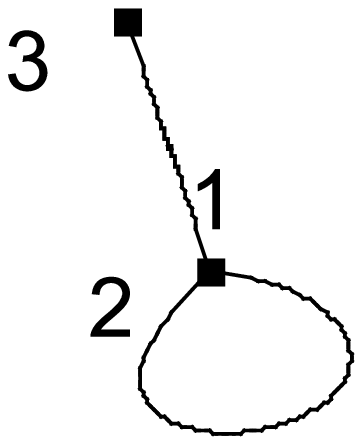}}
\scalebox{.30}{\includegraphics{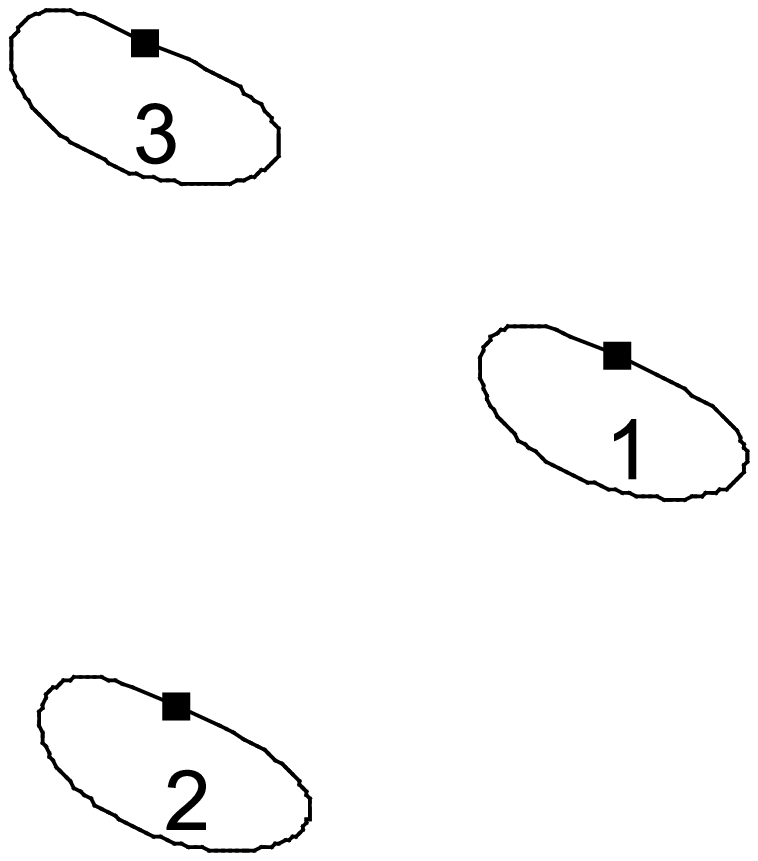}}
\caption{Four examples of small partial maps: 1) $(id_1, id_1)$; 2) $((12),(12))$; 3) $((12),(123))$; 4) $(id_3,id_3)$ }
\label{fig10}
\end{center}
\end{figure}

The drawer used in this article was built by Paulis \c{K}ikusts. The experience used here was typical for graph drawing area, see e.g., \cite{ki 96,fre 01,fre 02}. See pictures of combinatorial maps and partial maps in figures \ref{fig7}, \ref{fig6}, \ref{fig8}, \ref{fig10}, \ref{fig11} and \ref{fig12}.

In fig. \ref{fig10} examples of simplest partial maps are given. In fig. \ref{fig6} an example of {\em 4-critical multigraph} from \cite{ze 12} implemented planar and coded as map is drawn. Here, map is normalized with respect to its knot, and green and red cycles that cover corner set are clearly seen. In fig. \ref{fig7} drawing examples of two arbitrary generated maps are seen.

\begin{figure}[h]
\begin{center}
\scalebox{.22}{\includegraphics{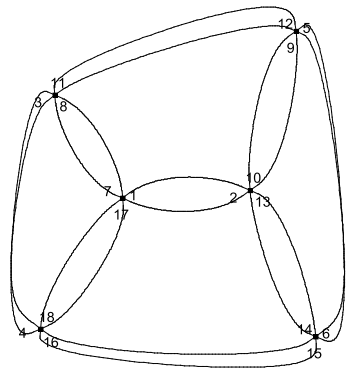}}
\caption{Prism map ($P$) with cut out faces as example of partial map $(P,\pi)$.}
\label{fig11}
\end{center}
\end{figure}

Similarly as in rotational case we are to choose two notions  from triple - points, lines and areas - to be varied independently. Tradition says that we leave areas to be determined by first two. Thus, our rotational prevalence doesn't exceed this border of point-line-area geometry. The break into this area would need new ideas. By the way, on simple way of uniforming these three geometrical quantities in a geometry of areas is suggested by cubic combinatorial maps in the approach of C.H.C. Little, see \cite{lit 88,bo li 95}.

\begin{figure}[h]
\begin{center}
\scalebox{.30}{\includegraphics{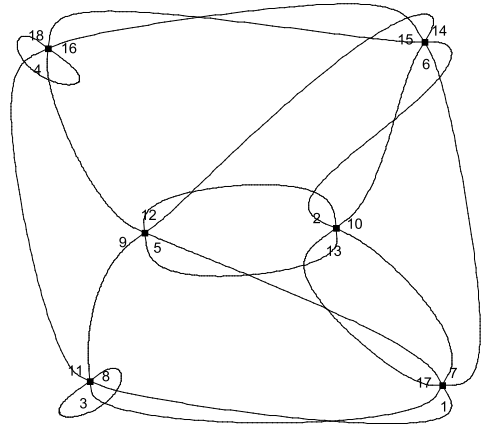}}
\caption{Partial map with two hyperedges of size 5 and 13 correspondingly defined by pair of permutations  $(1 \bar{7} 7 )( 2 \bar{0} \bar{3} )( 3 \bar{1} 8 )( 4 \bar{8} \bar{6} )( 6 \bar{5} \bar{4} )( 5 9 \bar{2})$ and $(1 3 8 )( 2 \bar{2} \bar{4} )( 4 \bar{1} 9 )( \bar{5} \bar{6} \bar{8} )( 5 \bar{3} \bar{7} )( 7 6 \bar{0})$. }
\label{fig12}
\end{center}
\end{figure}

\subsection{Combinatorial maps as combinatorial structure of algorithms and data structures}

One way to widen the application of rotational relations would be to remember that rotations work as sort of theorems starting from rotational data structures, cycles, graphs, constellations, and going on to operations and rotational theorems, and more daring theoretical principles, as in the case of Riemann geometry. It may find applications in how we build computer algorithms and before all the graph algorithms and maybe starting with implementing permutational calculus itself. First author of this article tried to use such philosophy in programming of a graph algorithm with but some practical outcome in somewhat good organized program.

However, rotational prevalence need be supported with theoretical investigations. We need to develop combinatorial map theory itself. Application of partial maps particularly might be useful there. Because of this possibility to visualize maps and especially partial maps would be very useful in these efforts.

\section{Conclusions}

Combinatorial map approach as purely combinatorial may serve as good example for graph theory and applications that combinatorial approaches may be used not only as theoretical support but as direct operational tools as well. By building operational outset of some systems combinatorial base may be used not only as required invariants of the elements of system, but the base for its structure itself already on level of its design and implementation. In this approach we suggest to use rotations-before-edges-faces-and-vertices approach as some rotational prevalence principle. We ask how we could apply this rotational prevalence in real conditions of graph drawing systems.

\section{Acknowledgements}

 We thank doctoral students of second author  Rihards Opmanis and Rudolfs Opmanis for their help in building map drawer program.

\end{document}